\documentclass{amsart}

\usepackage[all]{xypic}
\usepackage[centertags]{amsmath}
\usepackage{amsfonts}
\usepackage{amscd}
\usepackage{amssymb}
\usepackage{amsthm}
\usepackage{newlfont}
\usepackage{amsxtra}
 \newtheorem{thm}{Theorem}[section]
 \newtheorem{cor}[thm]{Corollary}
 
 \newtheorem{lem}[thm]{Lemma}
 
 \theoremstyle{definition}
 \newtheorem{defn}[thm]{Definition}
 \theoremstyle{remark}
 \newtheorem{rem}[thm]{Remark}
 \theoremstyle{remark}
 
 \theoremstyle{definition}
 
 \numberwithin{equation}{section}

 \newcommand{\Spec}{\mathrm{Spec}}
 \newcommand{\Frob}{\mathrm{Frob}}

 \newcommand{\Aut}{\mathrm{Aut}}
 \newcommand{\End}{\mathrm{End}}
 \newcommand{\Pic}{\mathrm{Pic}}
 
 \newcommand{\ord}{\mathrm{ord}}
 
 \newcommand{\Vol}{\mathrm{Vol}}
 \newcommand{\Gal}{\mathrm{Gal}}
 \newcommand{\GL}{\mathrm{GL}}
 \newcommand{\PGL}{\mathrm{PGL}}
 \newcommand{\SL}{\mathrm{SL}}

 \newcommand{\pr}{\mathrm{pr}}

 \newcommand{\Id}{\mathrm{Id}}

 \newcommand{\inv}{\mathrm{inv}}
 
 \newcommand{\Nr}{\mathrm{Nr}}

 \newcommand{\fc}{\mathfrak c}
 \newcommand{\fh}{\mathfrak h}
 
 \newcommand{\fp}{\mathfrak p}
 \newcommand{\fr}{\mathfrak r}
 
 \newcommand{\fn}{\mathfrak n}
 \newcommand{\fm}{\mathfrak m}

 \newcommand{\fS}{\mathfrak S}

 \newcommand{\cO}{\mathcal{O}}

 \newcommand{\cK}{\mathcal{K}}

 \newcommand{\cG}{\mathcal{G}}
 
 \newcommand{\cS}{\mathcal{S}}

 \renewcommand{\cD}{\mathcal{D}}
 \newcommand{\cE}{\mathcal{E}}
 
 \newcommand{\cF}{\mathcal{F}}
 
 \newcommand{\cI}{\mathcal{I}}

 \newcommand{\R}{\mathbb{R}}
 \newcommand{\C}{\mathbb{C}}
 \newcommand{\E}{\mathbb{E}}
 \newcommand{\F}{\mathbb{F}}
 \newcommand{\M}{\mathbb{M}}
 \newcommand{\Q}{\mathbb{Q}}
 \newcommand{\Z}{\mathbb{Z}}
 \newcommand{\A}{\mathbb{A}}

 \renewcommand{\P}{\mathbb{P}}

 \newcommand{\Ell}{\mathbf{Ell}}

 \newcommand{\To}{\longrightarrow}
 
 \newcommand{\bs}{\setminus}
 \newcommand{\bD}{\bar{D}}
 \newcommand{\tf}{\widetilde{F}}
 \newcommand{\tp}{\widetilde{\Pi}}

 \newcommand{\Fi}{F_\infty}

 \newcommand{\G}{\Gamma}

 \newcommand{\norm}[1]{\| #1\|}

 \newcommand{\twist}[1]{{^\tau}\!#1}

\begin{document}

\title{Genus formula for modular curves of $\cD$-elliptic sheaves}

\author{Mihran Papikian}

\address{Department of Mathematics, Pennsylvania State University, University Park, PA 16802}

\email{papikian@math.psu.edu}

\thanks{The author was supported in part by NSF grant DMS-0801208 and Humboldt Research Fellowship.}




\begin{abstract}
We prove a genus formula for modular curves of $\cD$-elliptic
sheaves. We use this formula to show that the reductions of modular
curves of $\cD$-elliptic sheaves attain the Drinfeld-Vladut bound as
the degree of the discriminant of $\cD$ tends to infinity.
\end{abstract}


\maketitle


\section{Introduction} The genus formulas for classical modular
curves and Shimura curves are well-known and quite useful in many
arithmetic problems.

Drinfeld modular curves \cite{Drinfeld} play a role over function
fields of positive characteristic similar to classical modular
curves over $\Q$. Many invariants of Drinfeld modular curves,
including their genera, were calculated by Gekeler in the 80's, cf.
\cite{GekelerLNM}, \cite{Invariants}. The function field analogue of
Shimura curves was introduced by Laumon, Rapoport and Stuhler in
\cite{LRS}. These curves are moduli spaces of certain objects,
called $\cD$-elliptic sheaves, which generalize the notion of
Drinfeld modules. The primary purpose of the present paper is to
produce a genus formula for modular curves of $\cD$-elliptic sheaves
(Theorem \ref{thmGen}). We also compute some other invariants of
these curves, such as the number of supersingular and elliptic
points. It turns out, maybe not surprisingly, that the genus formula
for modular curves of $\cD$-elliptic sheaves is strikingly similar
to the genus formula for Shimura curves. In the final section of the
paper we discuss an application of main results: we construct a new
sequence of curves over finite fields which is asymptotically
optimal. In view of the fact that only a few general families of
such sequences are known, this result is of independent interest.

The paper is organized as follows: In Section \ref{Sec1}, we fix the
notation and conventions which are used throughout the article. In
Section \ref{Sec2}, we recall the definition of $\cD$-elliptic
sheaves and the main geometric properties of their moduli schemes.
In Section \ref{Sec3}, we study the points on modular curves of
$\cD$-elliptic sheaves over finite fields. This relies on an
analogue of the Honda-Tate theory developed in \cite{LRS}. Here we
compute the number of $\cD$-elliptic sheaves having extra
automorphisms or large endomorphism algebras. These calculations are
used in the proof of the genus formula, and also in the construction
of asymptotically optimal sequences of curves. In Section
\ref{Sec4}, we compute the genus of modular curves of $\cD$-elliptic
sheaves. Finally, in Section \ref{Sec5}, we recall the definition of
asymptotically optimal sequences of curves over finite fields and
discuss how to construct such sequences using the modular curves of
$\cD$-elliptic sheaves.

\subsection*{Acknowledgments} I thank Professor E.-U. Gekeler for many
stimulating conversations on the subject of this article. The
article was written while I was visiting Saarland University. I
thank the department of mathematics for its hospitality.

\section{Notation}\label{Sec1}

Let $\F_q$ be the finite field with $q$ elements. Let
$C:=\P^1_{\F_q}$ be the projective line over $\F_q$. Denote by
$F=\F_q(T)$ the field of rational functions on $C$. The set of
closed points on $C$ (equivalently, places of $F$) is denoted by
$|C|$. For each $x\in |C|$, we denote by $\cO_x$ and $F_x$ the
completions of $\cO_{C,x}$ and $F$ at $x$, respectively. The residue
field of $\cO_x$ is denoted by $\F_x$, the cardinality of $\F_x$ is
denoted by $q_x$, the degree $m$ extension of $\F_x$ is denoted by
$\F_x^{(m)}$, and $\deg(x):=\dim_{\F_q} (\F_x)$. We assume that the
valuation $\ord_x:F_x\to\Z$ is normalized by $\ord_x(\varpi_x)=1$,
where $\varpi_x$ is a fixed uniformizer of $\cO_x$; the norm
$|\cdot|_x$ on $F_x$ is $q_x^{-\ord_x(\cdot)}$. Denote the adele
ring of $F$ by $\A:=\prod'_{x\in |C|} F_x$.

\vspace{0.1in}

Let $\infty:=1/T$ and $A:=\F_q[T]$ be the subring of $F$ consisting
of functions which are regular away from $\infty$. The completion of
an algebraic closure of $F_\infty$ is denoted $\C_\infty$. Fox each
$x\in |C|-\infty$, we denote by $\fp_x\lhd A$ the corresponding
prime ideal of $A$. For an ideal $\fn\lhd A$, we define its degree
as $\deg(\fn):=\dim_{\F_q}(A/\fn)$. Note that $\deg(\fp_x)=\deg(x)$.
The adele ring outside of $\infty$ is $\A^\infty:=\prod'_{x\in
|C|-\infty}F_x$. In particular, $\A=\A^\infty\times F_\infty$.

\vspace{0.1in}

Given a ring $H$, we denote by $H^\times$ the group of its units.

\vspace{0.1in}

Let $D$ be a quaternion algebra over $F$, i.e., a $4$-dimensional
central simple algebra over $F$. For $x\in |C|$, we let
$D_x:=D\otimes_F F_x$. We assume throughout the paper that $D$ is
split at $\infty$, i.e., $D_\infty\cong \M_2(F_\infty)$. (Here
$\M_2$ is the ring of $2\times 2$ matrices.) Let $R$ be the set of
places where $D$ is ramified. It is a well-known fact that the
cardinality of $R$ is even, and conversely, for any choice of a
finite set $R\subset |C|$ of even cardinality there is a unique
quaternion algebra ramified exactly at the places in $R$; see
\cite[p. 74]{Vigneras}. If $R\neq \emptyset$, then $D$ is a division
algebra; if $R=\emptyset$, then $D\cong \M_2(F)$. For $R\neq
\emptyset$, define the ideal $\fr:=\prod_{x\in R}\fp_x$. Let
$D^\times$ be the algebraic group over $F$ defined by
$D^\times(B)=(D\otimes_F B)^\times$ for any $F$-algebra $B$; this is
the multiplicative group of $D$.

\vspace{0.1in}

For a closed subscheme $I$ of $C$ with ideal sheaf $\cI$, let
$\cO_I:=\cO_C/\cI$. Fix a locally free sheaf $\cD$ of
$\cO_C$-algebras with stalk at the generic point equal to $D$ and
such that $\cD_x:=\cD\otimes_{\cO_C}\cO_x$ is a maximal order in
$D_x$. Denote $\cD^\infty:=\prod_{x\in |C|-\infty}\cD_x$. For a
finite nonempty closed subscheme $I$ of $C-R-\infty$, let
$$
\cD_I:=\cD\otimes_{\cO_C}\cO_I
\quad\text{and}\quad\cK_I^\infty:=\ker((\cD^\infty)^\times\to
\cD_I^\times).
$$

\section{$\cD$-elliptic sheaves and their moduli
schemes}\label{Sec2}

Let $W$ be an $\F_q$-scheme. Denote by $\Frob_W$ its Frobenius
endomorphism, which is the identity on the points and the $q$-th
power map on the functions. Denote by $C\times W$ the fibred product
$C\times_{\Spec(\F_q)}W$. For a sheaf $\cF$ on $C$ and $\cG$ on $W$,
the sheaf $\pr_1^\ast(\cF)\otimes \pr_2^\ast(\cG)$ on $C\times W$ is
denoted by $\cF\boxtimes \cG$.

Let $z:W\to C$ be a morphism of $\F_q$-schemes such that
$z(W)\subset C-R-\infty$. A \textit{$\cD$-elliptic sheaf over $W$},
with pole $\infty$ and zero $z$, is a sequence
$\E=(\cE_i,j_i,t_i)_{i\in \Z}$, where each $\cE_i$ is a locally free
sheaf of $\cO_{C\times W}$-modules of rank $4$ equipped with a right
action of $\cD$ compatible with the $\cO_C$-action, and where
\begin{align*}
j_i &:\cE_i\to \cE_{i+1}\\
t_i &:\twist{\cE}_{i}:=(\Id_C\times \Frob_{W})^\ast \cE_i\to
\cE_{i+1}
\end{align*}
are injective $\cO_{C\times W}$-linear homomorphisms compatible with
the $\cD$-action. The maps $j_i$ and $t_i$ are sheaf modifications
at $\infty$ and $z$, respectively, which satisfy certain conditions.
We refer to \cite[$\S$2]{LRS} for the precise definition.

Let $I$ be a closed subscheme of $C-R-\infty$. Let
$\E=(\cE_i,j_i,t_i)$ be a $\cD$-elliptic sheaf over $W$. Assume
$z(W)$ is disjoint from $I$. The restriction $\cE_I:=\cE_{i|I\times
W}$ is independent of $i$, and $t$ induces an isomorphism
$^\tau\!\cE_I\cong \cE_I$. A \textit{level-$I$ structure} on $\E$ is
an $\cO_{I\times W}$-linear isomorphism $\iota: \cD_I\boxtimes
\cO_W\cong \cE_I$, compatible with the action of $\cD_I$, which
makes the following diagram commutative:
$$
\xymatrix{ ^\tau\!\cE_I\ar[rr]^-t & &\cE_I\\
& \cD_I\boxtimes \cO_W \ar[lu]^-{^\tau\!\iota}\ar[ru]_-\iota&}
$$

Denote by $\Ell^\cD_I(W)$ the category whose objects are the
$\cD$-elliptic sheaves over $W$ with level $I$-structures and whose
morphisms are isomorphisms. (If $I=\emptyset$, then $\Ell^\cD_I(W)$
is simply the category of $\cD$-elliptic sheaves over $W$.) There is
a natural free action of $\Z$ on $\Ell^\cD_I(W)$: $n\in \Z$ acts by
$$
[n](\cE_i,j_i,t_i;\iota)=(\cE_{i+n},j_{i+n},t_{i+n};\iota).
$$
We denote the quotient category by $\Ell^\cD_I(W)/\Z$. There is also
an action of $\cD_I^\times$ on $\Ell_{I}^\cD(W)$: $g\in
\cD_I^\times$ acts by
$$
(\cE_i,j_i,t_i;\iota)g=(\cE_{i},j_{i},t_{i};\iota\circ g),
$$
where $g$ acts on $\cD_I\boxtimes \cO_W$ via right multiplication on
$\cD_I$. The actions of $\Z$ and $\cD_I^\times$ obviously commute
with each other.

We have the following fundamental theorem:

\begin{thm} Fix some $I\neq \emptyset$.
The functor $W\mapsto \Ell^{\cD}_{I}(W)/\Z$ is representable by a
smooth quasi-projective scheme $X^\cD_I$ over $C':=C-I-R-\infty$ of
pure relative dimension $1$. If $D$ is a division algebra, then
$X^\cD_I$ is proper over $C'$.
\end{thm}
\begin{proof}
This is a consequence of (4.1), (5.1) and (6.2) in \cite{LRS}.
\end{proof}

\begin{cor}\label{cor2.3}
The functor $W\mapsto \Ell^{\cD}_\emptyset(W)/\Z$ has a coarse
moduli scheme $X^\cD_\emptyset$, which is smooth of pure relative
dimension $1$ over $C-R-\infty$ with geometrically irreducible
fibres. If $D$ is a division algebra, then $X^\cD_\emptyset$ is
proper. For $I\neq \emptyset$, there is a functorial morphism
$X^\cD_I\to X^\cD_\emptyset$ over $C'$ which is finite flat of
degree $\#(\cD_I^\times)/(q-1)$.
\end{cor}
\begin{proof} The action of $\cD_I^\times$ on the stack $\Ell^\cD_I$
induces an action of $\cD_I^\times$ on the scheme $X^\cD_I$. The
quotient scheme $X^\cD_I/\cD_I^\times$ is the coarse moduli scheme
for $W\mapsto\Ell^{\cD}_\emptyset(W)/\Z$ over $C-I-R-\infty$. The
subgroup of $\cD_I^\times\cong \GL_2(\cO_I)$ which fixes all
isomorphism classes of objects in $\Ell^\cD_I(W)$ is $\F_q^\times$
embedded diagonally into $\GL_2(\cO_I)$. Hence the morphism
$X^\cD_I\to X^\cD_I/\cD_I^\times$ is finite flat of degree
$\#(\cD_I^\times)/(q-1)$. To extend $X^\cD_\emptyset$ over $I$, take
$J$ disjoint from $I$. The quotient $X^\cD_J/\cD_J^\times$ is
defined over $C-J-R-\infty$ and is isomorphic to
$X^\cD_I/\cD_I^\times$ over $C-J-I-R-\infty$. One obtains
$X^\cD_\emptyset$ by gluing these two schemes over $C-J-I-R-\infty$.
Finally, the geometric fibres of $$X^\cD_\emptyset\to C-R-\infty$$
are irreducible by \cite[Prop. 3.2]{PapMRL}.
\end{proof}

\begin{rem}\label{remDM}
Using the Morita equivalence and a result of Drinfeld, one can show
that when $D=\M_2$ the category $\Ell^{\cD}_{I}(W)/\Z$ is equivalent
to the category of rank-$2$ Drinfeld $A$-modules over $W$ equipped
with level $I$-structures, cf. \cite{Carayol}. Hence in that case
$X^\cD_I$ are the modular curves constructed in \cite{Drinfeld}.
\end{rem}

Assume $D$ is a division algebra. Denote by
$X^\cD_{I,F}:=X^\cD_{I}\times_{C'}\Spec(F)$ the generic fibre of
$X^\cD_I$. For a closed point $o$ of $C-I-R-\infty$ denote
$X^\cD_{I,o}:=X^\cD_{I}\times_{C'}\Spec(\F_o)$. Let $\bar{F}$ (resp.
$\overline{\F}_o$) be a fixed algebraic closure of $F$ (resp.
$\F_o$). Fix a prime number $\ell$ not equal to the characteristic
of $F$. Consider the $\ell$-adic cohomology groups
$$
H^i_{I,F}:=H^i(X^{\cD}_{I,F}\otimes_F \bar{F}, \overline{\Q}_\ell)
$$
and
$$
H^i_{I,o}:=H^i(X^{\cD}_{I,o}\otimes_{\F_o}\overline{\F}_o,
\overline{\Q}_\ell).
$$
These are finite dimensional $\overline{\Q}_\ell$-vector spaces
which are non-zero only if $0\leq i\leq 2$. The dimensions of these
spaces do not depend on the choice of $\ell$. As a consequence of
the proper base change theorem, there are canonical isomorphisms of
$\overline{\Q}_\ell$-vector spaces $H^i_{I,F}\cong H^i_{I,o}$,
$0\leq i\leq 2$, cf. \cite[p. 287]{LRS}. The
\textit{Euler-Poincar\'e characteristic} of $X^\cD_{I,F}$ is
$$
\chi(X^\cD_{I,F})=\sum_{i=0}^2
(-1)^i\dim_{\overline{\Q}_\ell}(H^i_{I,F}).
$$
Similarly, define
$$
\chi(X^\cD_{I,o})=\sum_{i=0}^2
(-1)^i\dim_{\overline{\Q}_\ell}(H^i_{I,o}).
$$
$\chi(X^\cD_{I,o})$ is independent of the choice of $o$, since it is
equal to $\chi(X^\cD_{I,F})$.

\section{$\cD$-elliptic sheaves of finite
characteristic}\label{Sec3}

Let $o\in |C|-R-\infty$ be fixed, and let $k$ be a fixed algebraic
closure of $\F_o$. A \textit{$\cD$-elliptic sheaf of characteristic
$o$ over $k$} is a $\cD$-elliptic sheaf $\E$ over $\Spec(k)$ such
that the zero is the canonical morphism
$$
z:\Spec(k)\to \Spec(\F_o)\hookrightarrow C.
$$

In \cite[Ch. 9]{LRS}, the authors, following Drinfeld, develop a
Honda-Tate type theory for $\cD$-elliptic sheaves. The basis of this
theory is a construction which attached to each $\cD$-elliptic sheaf
of characteristic $o$ over $k$ a $(D,\infty,o)$-type. A
\textit{$(D,\infty,o)$-type} is a pair $(\tf,\tp)$ having the
following properties, cf. \cite[(9.11)]{LRS}:
\begin{enumerate}
\item $\tf$ is a separable field extension of $F$ of degree dividing
$2$;
\item $\tp\in \tf^\times\otimes_\Z \Q$, but $\tp\not\in F^\times\otimes_\Z
\Q$ unless $\tf=F$;
\item $\infty$ does not split in $\tf$;
\item The valuations of $F$ naturally extend to the group $\tf^\times\otimes_\Z\Q$.
There are exactly two places $\tilde{x}=\tilde{\infty}$ and
$\tilde{x}=\tilde{o}$ of $\tf$ such that $\ord_{\tilde{x}}(\Pi)\neq
0$. The place $\tilde{\infty}$ is the unique place of $\tf$ over
$\infty$ and $\tilde{o}$ divides $o$. Moreover,
$$
\ord_{\tilde{\infty}}(\tp)\cdot \deg(\tilde{\infty})=-[\tf:F]/2.
$$
\item For each place $x$ of $F$ and each place $\tilde{x}$ of $\tf$
dividing $x$, we have
$$
(2[\tf_{\tilde{x}}:F_x]/[\tf:F])\cdot \inv_x(D)\in \Z.
$$
\item $h:=2[\tf_{\tilde{o}}:F_o]/[\tf:F]$ is an integer satisfying $1\leq h\leq
2$.
\end{enumerate}

\vspace{0.1in}

Next, using (1)-(6), we deduce some additional properties of $\tf$
depending on whether $h=1$ or $2$.

\begin{lem}\label{lem3.1}
If $h=1$, then $\tf$ is a separable quadratic extension of $F$ in
which $o$ splits and each $x\in R$ does not split. If $h=2$, then
$\tf=F$
\end{lem}
\begin{proof}
Let $h=1$. Since $1=2[\tf_{\tilde{o}}:F_o]/[\tf:F]$ and $[\tf:F]\leq
2$, we must have $[\tf:F]=2$ and $[\tf_{\tilde{o}}:F_o]=1$. From
$[\tf_{\tilde{x}}:F_x]\cdot \inv_x(D)\in \Z$, we conclude that
$[\tf_{\tilde{x}}:F_x]=2$ for every $x\in R$.

Now let $h=2$. Suppose $\tf\neq F$. Since
$2=2[\tf_{\tilde{o}}:F_o]/[\tf:F]$, we must have
$[\tf_{\tilde{o}}:F_o]=2$. Hence $\tilde{o}$ is the unique place of
$\tf$ over $o$. Let $g$ be the generator of $\Gal(\tf/F)$. Let $N$
be a non-zero integer such that $\tp^N\in \tf$. Consider
$F':=F[\tp^N]$. Since $\tp\in F'^\times\otimes \Q$, we must have
$F'=\tf$. Since $\tp^N$ has non-zero valuations only at $\tilde{o}$
and $\tilde{\infty}$, and $g$ fixes both places as neither of them
splits, $\tp^N/g(\tp^N)$ has zero valuation at all places of $\tf$.
This implies that $\tp^N/g(\tp^N)$ is a root of unity, so there is
an integer $M\geq N$ such that $\tp^M/g(\tp^M)=1$. By the same
argument as above, $\tf=F[\tp^M]$. But now $g$ fixes $\tp^M$, so
$\tf=F$, which is a contradiction.
\end{proof}

\begin{defn} Following the terminology for
elliptic curves, we call a $\cD$-elliptic sheaf of characteristic
$o$ over $k$ \textit{ordinary} or \textit{supersingular} depending
on whether its $(D,\infty,o)$-type has $h=1$ or $h=2$, respectively.
Denote by $\fS^\cD_o\subset X^\cD_{\emptyset, o}(k)$ the subset of
closed points corresponding to supersingular $\cD$-elliptic sheaves.
\end{defn}

To an $\E$ with $(D,\infty,o)$-type $(\tf,\tp)$ there is a
canonically associated central division algebra $\End^0(\E)$ over
$\tf$, called the \textit{endomorphism algebra of $\E$}; see
\cite[(9.10)]{LRS}. There is also a natural notion of
\textit{endomorphism ring} $\End(\E)$ of $\E$ \cite{Papikian}, which
is an $A$-order in $\End^0(\E)$.

As is proven in \cite{LRS}, if $\E$ is ordinary, then
$\End^0(\E)=\tf$. If $\E$ is supersingular, then $\End^0(\E)$ is the
quaternion algebra $\bD$ over $F$ with invariants
$$
\inv_{x}\bD=\left\{
               \begin{array}{ll}
                 1/2, & \hbox{if $x=\infty$;} \\
                 1/2, & \hbox{if $x=o$;} \\
                 \inv_x(D), & \hbox{otherwise.}
               \end{array}
             \right.
$$

\begin{defn}
Let $\Aut(\E)$ be the set of all non-zero elements of $\End(\E)$
which are algebraic over $\F_q$ (cf. \cite[$\S$6]{Papikian}).
\end{defn}

\begin{lem}
The elements of $\Aut(\E)$ form a group isomorphic to either
$\F_q^\times$ or $\F_{q^2}^\times$.
\end{lem}
\begin{proof}
The statement is obvious when $\E$ is ordinary. If $\E$ is
supersingular, then this follows from \cite{Papikian}.
\end{proof}

\begin{defn}
We say that \textit{$\E$ has extra automorphisms} if $\Aut(\E)\cong
\F_{q^2}^\times$.
\end{defn}

To proceed further, we need to recall some standard facts about
maximal orders in quaternion algebras. The definitions and the
proofs can be found in \cite{Vigneras}. Fix some maximal $A$-order
$\cO$ in $\bD$. Let $\cS:=\{I_1,\dots, I_m\}$ be the right ideal
classes of $\cO$, and let $\cO_i:=O_\ell(I_i)$ be the left order of
$I_i$, $1\leq i\leq m$. The number $\fc(\bD, \infty):=m$, called the
\textit{class number} of $\bD$ with respect to $\infty$, does not
depend on the choice of $\cO$. Moreover, the orders $\cO_i$, $1\leq
i\leq m$, are maximal, and all conjugacy classes of maximal
$A$-orders in $\bD$ appear in $\{\cO_1,\dots, \cO_m\}$. The
\textit{mass of $\bD$} with respect to $\infty$ is
$$
\fm(\bD,\infty):=(q-1)\sum_{i=1}^m (\#\cO_i^\times)^{-1}.
$$
Finally, the two-sided ideals of $\cO$ form a commutative group
generated by the ideals of $A$ and ideals $\Pi_x$, $x\in R\cup o$,
such that $\Pi_x^2=\fp_x$ (see \cite[pp. 86-87]{Vigneras}).

\vspace{0.1in}

For a finite set $S\subset |C|$, let
$$
\wp(S)=\left\{
         \begin{array}{ll}
           0, & \hbox{if some place in $S$ has even degree;} \\
           1, & \hbox{otherwise.}
         \end{array}
       \right.
$$
In particular, $\wp(\emptyset)=1$.
\begin{thm}\label{thmDvG}
With above notation,
$$
\fm(\bD,\infty)=\frac{1}{q^2-1}\prod_{x\in R\cup o}(q_x-1)
$$
and
$$
\fc(\bD,\infty)=\frac{1}{q^2-1}\prod_{x\in R\cup
o}(q_x-1)+\frac{q}{q+1}\cdot 2^{\# R}\cdot \wp(R\cup o).
$$
\end{thm}
\begin{proof}
See (1), (4) and (6) in \cite{DvG}.
\end{proof}

\begin{thm}\label{thmPapLRS} With above notation, there exists a bijection
$$
\fS^\cD_o\overset{\sim}{\To} \cS
$$
such that if $\E_j\in \fS^\cD_o$ corresponds to $I_j\in \cS$ then
$\End(\E_j)\cong O_\ell(I_j)$. The Galois group $\Gal(k/\F_o)$
preserves the set $\fS^\cD_o$. Moreover, the action of the geometric
Frobenius $\Frob_o$ on $\fS^\cD_o$ corresponds to the action of
$\Pi_o$ on $\cS$ given by right multiplication $I_j\mapsto
I_j\Pi_o$.
\end{thm}
\begin{proof} First note
that since $O_r(I_j)=O_\ell(\Pi_o)$ the product $I_j\Pi_o$ makes
sense, and it is a right ideal of $\cO$, cf. \cite[p. 22]{Vigneras}.
For the proof of the first statement of the theorem see
\cite[$\S$7]{Papikian}. (When $D$ is the matrix algebra, i.e., in
the case of Drinfeld modules, this is due to Gekeler \cite[Thm.
4.3]{GekelerFDM}.) The second statement can be deduced from the
discussion in \cite[$\S$10]{LRS} (see p. 274, in particular).
\end{proof}

\begin{cor}\label{thmSS}
All points in $\fS^\cD_o$ are rational over $\F_o^{(2)}$, and
$$
\# \fS^\cD_o=\frac{1}{q^2-1}\prod_{x\in R\cup
o}(q_x-1)+\frac{q}{q+1}\cdot 2^{\# R}\cdot \wp(R\cup o).
$$
\end{cor}
\begin{proof} The formula for $\# \fS^\cD_o$ is an immediate consequence of
Theorems \ref{thmDvG} and \ref{thmPapLRS}. For the first statement,
note that the action of $\Frob_o^2$ on $\fS^\cD_o$ corresponds to
the action of $\fp_o$ on the ideal classes $\cS$. But $\fp_o$ is a
principal ideal of $A$ since $\Pic(A)=1$, so it acts trivially on
the ideal classes.
\end{proof}

\begin{cor}\label{thmEP}
The number of points on $X^\cD_{\emptyset, o}(k)$ corresponding to
$\cD$-elliptic sheaves with extra automorphisms is $2^{\# R}\cdot
\wp(R)$. Moreover, when $\wp(R)\neq 0$, all $\cD$-elliptic sheaves
with extra automorphisms are ordinary (resp. supersingular) if
$\deg(o)$ is even (resp. odd).
\end{cor}
\begin{proof} Let $w$ be the number of points on $X^\cD_{\emptyset, o}(k)$ corresponding to
$\cD$-elliptic sheaves with extra automorphisms. Let $\E\in
\fS^\cD_o$, and let $E\in \cS$ be the ideal corresponding to $\E$
under the bijection of Theorem \ref{thmPapLRS}. Since
$\Aut(\E)=O_\ell(E)^\times$, the number of points in $\fS^\cD_o$
with extra automorphisms is equal to
$$
\left(\fc(\bD,\infty)-\fm(\bD,\infty)\right)/\left(1-(q+1)^{-1}\right),
$$
which is $2^{\# R}\cdot \wp(R\cup o)$, according to Theorem
\ref{thmDvG}. Now suppose $\E$ is ordinary. If $\Aut(\E)\cong
\F_{q^2}^\times$, then $\F_{q^2}$ embeds into $\tf$. Since $\tf$ is
quadratic, we must have $\tf=\F_{q^2}F$. On the other hand, by Lemma
\ref{lem3.1}, the places in $R$ do not split in $\tf$, but $o$
splits. Hence the degree of each place in $R$ must be odd and
$\deg(o)$ must be even. Overall, so far, we can make the following
conclusion. If some place in $R$ has even degree then $w=0$. If all
places in $R$ have odd degrees and $\deg(o)$ is also odd, then
$w=2^{\# R}$, since in this case the $\cD$-elliptic sheaves with
extra automorphisms are necessarily supersingular.

Now suppose $D$ is a division algebra. Fix some non-empty $I\subset
C-R-o-\infty$ and consider the covering $\pi:X^\cD_{I,o}\to
X^\cD_{\emptyset, o}$. The points which ramify in this covering are
exactly the points corresponding to $\cD$-elliptic sheaves with
extra automorphisms. Moreover, the ramification index at such a
point is equal to $\#(\F_{q^2}^\times/\F_q^\times)=q+1$. Hence the
ramifications are tame, and the Riemann-Hurwitz formula implies
\begin{equation}\label{eqRH}
\chi(X^\cD_{I,o})=\deg(\pi)\cdot\chi(X^\cD_{\emptyset,
o})-\deg(\pi)\frac{q}{q+1}w.
\end{equation}
As we mentioned in $\S2$, the Euler-Poincar\'e characteristic and
$\deg(\pi)$ do not depend on the choice of $o$, so $w$ also does not
depend on $o$. Therefore, if $\deg(o)$ is even, then $w=2^{\#
R}\cdot \wp(R)$ and all $\cD$-elliptic sheaves with extra
automorphisms are ordinary.

Finally, suppose $D\cong \M_2(F)$ (equiv. $R=\emptyset$). Then the
problem can be reformulated as a problem about rank-$2$ Drinfeld
$A$-modules over $k$, cf. Remark \ref{remDM}. In those terms, the
problem was solved by Gekeler using different techniques; cf.
\cite[Prop. 7.1]{GekelerGQ}: There is a unique $k$-isomorphism class
of rank-$2$ Drinfeld $A$-modules with automorphism group
$\F_{q^2}^\times$. The $j$-invariant of this class is $0$, and a
Drinfeld module with $j=0$ is supersingular or ordinary, depending
on whether $\deg(o)$ is odd or even; see \cite{Invariants},
\cite{GekelerGQ}.
\end{proof}

\section{Genus formula}\label{Sec4} In this section we compute $\chi(X^\cD_{I,F})$
and $\chi(X^\cD_{\emptyset, F})$. We will compute the first number
using analytic methods, and then deduce the second number from the
Riemann-Hurwitz formula. Throughout the section we assume that $D$
is a division algebra.

\vspace{0.1in}

Let $I$ be a fixed closed non-empty subscheme of $C-R-\infty$. The
double coset space $D^\times(F)\bs
D^\times(\A^\infty)/\cK_{I}^\infty$ has finite cardinality. In fact,
since $D$ is split at $\infty$, the reduced norm induces a bijection
$$
D^\times(F)\bs D^\times(\A^\infty)/\cK_{I}^\infty
\overset{\sim}{\To}
F^\times\bs(\A^\infty)^\times/\Nr(\cK_{I}^\infty).
$$
(This is a consequence of the Strong Approximation Theorem for
$D^\times$, cf. \cite[p. 89]{Vigneras}.) Choose a system $S$ of
representatives for this double coset space. For each $s\in S$, let
$$
\G_{I,s}:=D^\times(F)\cap s \cK_{I}^\infty s^{-1}.
$$
\begin{lem}
Under the natural embedding
$$
\G_{I,s}\hookrightarrow
D^\times(F)/F^\times\hookrightarrow D^\times(\Fi)/\Fi^\times\cong
\PGL_2(\Fi),
$$
$\G_{I,s}$ is a discrete, cocompact, torsion-free
subgroup of $\PGL_2(\Fi)$.
\end{lem}
\begin{proof}
See \cite[Lem. 6.4]{PapCrelle}.
\end{proof}

Let $\Omega$ denote the Drinfeld upper half-plane over $\Fi$; see
\cite{Drinfeld}. As a set, $\Omega=\C_\infty-\Fi$. By the previous
lemma, each group $\G_{I,s}\subset \PGL_2(\Fi)$ is a Schottky
subgroup which acts on $\Omega$ via linear fractional
transformations. As follows from the theory of Mumford curves, the
quotient $\G_{I,s}\bs \Omega$ is the analytification of a smooth
projective curve $X_{\G_{I,s}}$ over $\C_\infty$. Denote by
$X^\cD_{I}(\C_\infty)$ the underlying rigid-analytic variety of
$X^\cD_{I,F}\otimes_F \Fi$. We have the following fundamental fact:
\begin{thm}
There is a canonical isomorphism of analytic spaces over $\C_\infty$
$$
X^\cD_{I}(\C_\infty)\cong \bigsqcup_{s\in S} X_{\G_{I,s}}.
$$
\end{thm}
\begin{proof}
See \cite[Thm. 4.4.11]{BS}.
\end{proof}

Now using the theory of rigid-analytic uniformizations of Jacobians
of Mumford curves \cite{MD}, one concludes that
$$
\chi(X^\cD_{I,F})=2\cdot \sum_{s\in S}\chi(\G_{I,s}),
$$
where $\chi(\G_{I,s})=1-\dim_{\Q}H^1(\G_{I,s}, \Q)$.

In \cite{SerreCGD}, Serre developed a theory which allows to compute
$\chi(\G_{I,s})$ as a volume. Serre's result is reproduced in a
convenient form in Proposition 5.3.6 of \cite{LaumonCDV}. Combining
this with Proposition 5.3.9 in \cite{LaumonCDV}, in our situation we
get the following statement:

Let $dg$ be the Haar measure on $\GL_2(\Fi)$ normalized by
$\Vol(\GL_2(\cO_\infty), dg)=1$. Let $dz$ be the Haar measure on
$\Fi^\times$ normalized by $\Vol(\cO_\infty^\times, dz)=1$. Fix the
Haar measure $dh=dg/dz$ on $\PGL_2(\Fi)$ and the counting measure
$d\delta$ of $\G_{I,s}$. Then
\begin{equation}\label{eq-3.5}
\chi(\G_{I,s})=\frac{1}{2}(1-q)\cdot \Vol\left(\G_{I,s}\bs
\PGL_2(\Fi), \frac{dh}{d\delta}\right).
\end{equation}

Now define a Haar measure $d\bar{g}$ on $D^\times(\A)$ as follows:
For $x\in |C|$, normalize the Haar measure $dg_x$ on $D^\times(F_x)$
by $\Vol(\cD_x^\times, dg_x)=1$, and let $d\bar{g}$ be the
restricted product measure. We have
$$
D^\times(F)\bs D^\times(\A)/\cK_{I}^\infty \Fi^\times \cong
\bigsqcup_{s\in S} \left(\G_{I,s}\bs \PGL_d(\Fi)\right),
$$
and the push-forward of $d\bar{g}$ on $D^\times(\A)$ to the double
coset space above induces the measure $dh/d\delta$ on each
$\G_{I,s}\bs \PGL_2(\Fi)$. Hence
$$
\chi(X^\cD_{I,F})=(1-q)\cdot \Vol\left(D^\times(F)\bs
D^\times(\A)/\cK_{I}^\infty \Fi^\times, d\bar{g}\right).
$$

Consider the homomorphism
$$
\norm{\cdot}:D^\times(\A)\to q^{\Z}
$$
given by the composition of the reduced norm $\Nr: D^\times(\A)\to
\A^\times$ with the idelic norm $\prod_{x\in |C|}|\cdot|_x:
\A^\times \to q^\Z$. Denote the kernel of this homomorphism by
$D^{1}(\A)$. The group $D^\times(F)$, under the diagonal embedding
into $D^\times(\A)$, lies in $D^{1}(\A)$, thanks to the product
formula. The quotient $D^\times(F)\bs D^1(\A)$ is compact, hence has
finite volume. It is well-known that $\norm{\cdot}:D^\times(\A)\to
q^{\Z}$ is surjective. The image of $\Fi^\times$ in $q^{\Z}$ is
clearly $q^{2\Z}$. Hence there is an exact sequence
$$
0\to D^\times(F)\bs D^1(\A)/\cO_\infty^\times\to D^\times(F)\bs
D^\times(\A)/\Fi^\times \to \Z/2\Z\to 0,
$$
which implies
$$
\Vol\left(D^\times(F)\bs D^\times(\A)/\cK_{I}^\infty \Fi^\times,
d\bar{g}\right)
$$
$$
= 2\cdot\# \GL_2(\cO_I)\cdot \Vol\left(D^\times(F)\bs D^1(\A),
d\bar{g}\right).
$$
This last volume can be expressed in terms of the zeta-function of
$C$ (see \cite[$\S$4]{PapCrelle}):
$$
\Vol\left(D^\times(F)\bs D^1(\A),
d\bar{g}\right)=\frac{1}{(q-1)^2(q^2-1)}\prod_{x\in R}(q_x-1).
$$

Combining the previous calculations, one obtains:
\begin{thm}\label{thmChiI}
$$
\chi(X^\cD_{I,F})=-\frac{2\cdot \# \GL_2(\cO_I)}{(q-1)(q^2-1)}\cdot
\prod_{x\in R}(q_x-1).
$$
\end{thm}

Now consider $\chi(X^\cD_{\emptyset, F})$. Fix some closed point $o$
on $C-I-R-\infty$. We know that $\chi(X^\cD_{\emptyset,
F})=\chi(X^\cD_{\emptyset, o})$, and $\chi(X^\cD_{I,
F})=\chi(X^\cD_{I, o})$. Hence, if we combine Corollary
\ref{cor2.3}, (\ref{eqRH}), Corollary \ref{thmEP} and Theorem
\ref{thmChiI}, then we obtain the formula:
\begin{equation}\label{eqChi}
\chi(X^\cD_{\emptyset, F})=-\frac{2}{(q^2-1)} \prod_{x\in
R}(q_x-1)+\frac{q}{q+1}\cdot 2^{\# R}\cdot \wp(R).
\end{equation}

By Corollary \ref{cor2.3}, $X^\cD_{\emptyset, F}$ is a smooth,
projective, geometrically irreducible curve over $F$. To simplify
the notation, denote this curve by $X^R$ and its fibre over $o$ by
$X^R_o$. Let $g(X^R)$ be the genus of $X^R$. The Euler-Poincar\'e
characteristic $\chi(X^R)$ and $g(X^R)$ are related by the formula
$\chi(X^R)=2-2\cdot g(X^R)$. Hence from (\ref{eqChi}) we get:
\begin{thm}\label{thmGen}
$$
g(X^R)=1+\frac{1}{q^2-1}\left(\prod_{x\in R}(q_x-1)-q\cdot
(q-1)\cdot 2^{\# R-1}\cdot \wp(R)\right).
$$
\end{thm}

\begin{cor}\label{corHE} $g(X^R)$ is always divisible by $q$. In particular, $X^R$ is never
an elliptic curve. $g(X^R)=0$ if and only if either $\deg(\fr)=2$,
or $q=4$ and $\fr=(T^4-T)$.
\end{cor}

\begin{rem}
One can also try to classify those $X^R$ which are hyperelliptic. In
this direction, we are able to prove that for a fixed $q$ there are
only finitely many $X^R$ which are hyperelliptic. If $q$ is odd,
then $X^R$ is hyperelliptic if and only if $\deg(\fr)=3$. The proofs
of these results will appear elsewhere.
\end{rem}

It is interesting to compare the formula in Theorem \ref{thmGen} to
the formula for the genus of Shimura curves over $\Q$. First, we
rewrite the formula for $g(X^R)$ in a slightly different form.

Let $K$ be a global field and $L$ be a separable quadratic extension
of $K$. For a place $x$ of $K$, define the \textit{Artin-Legendre
symbol} $\left(\frac{L}{x}\right)$ (cf. \cite[p. 94]{Vigneras}):
$$
\left(\frac{L}{x}\right)=\left\{
                           \begin{array}{ll}
                             1, & \hbox{if $x$ splits in $L$;} \\
                             -1, & \hbox{if $x$ is inert in $L$;} \\
                             0, & \hbox{if $x$ ramifies in $L$.}
                           \end{array}
                         \right.
$$

Using this notation, we have
$$
g(X^R)=1+\frac{1}{q^2-1}\prod_{x\in
R}(q_x-1)-\frac{1}{2}\frac{q}{q+1}\prod_{x\in
R}\left(1-\left(\frac{\F_{q^2}F}{x}\right)\right).
$$

Now let $H$ be the indefinite quaternion division algebra over $\Q$
with reduced discriminant $d$. This means that $H\otimes \R\cong
\M_2(\R)$ and $H$ is ramified exactly at the primes dividing $d$.
Let $\cO$ be a maximal $\Z$-order in $H$. Denote by $\G^d$ the group
of units of positive norm in $\cO$. Under the embedding
$$
\G^d\hookrightarrow H\otimes \R\cong \M_2(\R)
$$
the image of $\G^d$ lies in $\SL_2(\R)$, hence naturally acts on the
upper half-plane $\fh$. The quotient $\fh/\G^d$ is a smooth,
projective, geometrically irreducible curve $X^d$ which has a model
over $\Q$. (In fact, $X^d$ is a moduli space of principally
polarized abelian surfaces equipped with an action of $\cO$.) The
genus $g(X^d)$ is given by the formula (see \cite[p.
120]{Vigneras}):
$$
g(X^d)=1+\frac{1}{12}\prod_{p|d}(p-1)-\frac{1}{2}
\left(\frac{1}{2}\prod_{p|d}\left(1-\left(\frac{\Q(\sqrt{-1})}{p}\right)\right)+
\frac{2}{3}\prod_{p|d}\left(1-\left(\frac{\Q(\sqrt{-3})}{p}\right)\right)\right).
$$

We see that the formulas for $g(X^d)$ and $g(X^R)$ are strikingly
similar. Even the terms $1/12$ and $1/(q^2-1)$ are of similar
nature: $-\zeta_\Z(-1)=1/6$ and $-\zeta_A(-1)=1/(q^2-1)$, where
$\zeta_\Z(s)$ is the Riemann zeta-function of $\Z$ and $\zeta_A(s)$
is the zeta-function of $A$.

\section{Asymptotically optimal series of curves}\label{Sec5}

Let $X$ be a curve of genus $g(X)$ defined over $\F_q$. (In this
section, a \textit{curve} means smooth, projective, geometrically
irreducible algebraic curve.) According to a celebrated result of
Weil, the number of $\F_q$-rational points on $X$ is bounded by
$$
\# X(\F_q)\leq q+1+2g(X)\sqrt{q}.
$$
On the other hand, it turns out that when $g(X)$ is very large
compared to $q$, then Weil's bound can be significantly improved.
Drinfeld and Vladut \cite{DV} showed that
\begin{equation}\label{VDbound}
\underset{X}{\lim\mathrm{sup}}\left(\frac{\#
X(\F_q)}{g(X)}\right)\leq \sqrt{q}-1,
\end{equation}
where $X$ runs over the set of all curves over $\F_q$. A series of
curves $\{X_i\}_{i\in \mathbb{N}}$ is called \textit{asymptotically
optimal} if
$$
\underset{i\to
\infty}{\lim}\left(\frac{\#X_i(\F_q)}{g(X_i)}\right)=\sqrt{q}-1,
$$
or, in other words, $\{X_i\}_{i\in \mathbb{N}}$ realizes the bound
(\ref{VDbound}). If $q$ is not a square then no asymptotically
optimal series of curves is known. If $q$ is a square then
asymptotically optimal series always exist, but every known such
series has the property that for all sufficiently large $i$ the
curve $X_i$ is a modular curve of some sort.

In \cite{PapMRL}, we showed that if one fixes $D$ and considers the
modular curves of $\cD$-elliptic sheaves with appropriate level
structures, then one obtains asymptotically optimal series of curves
over $\F_{q^2}$ as the level varies. Using the genus formula proven
in the current paper, we can show that the sequence
``perpendicular'' to the one mentioned in the previous sentence is
also asymptotically optimal. Namely, we \textit{fix} the level
(actually $I=\emptyset$) and \textit{vary} $D$:

\begin{thm}
Let $o=(T)$ be fixed. Let $R$ run over all subsets of $|C|-o-\infty$
having even cardinality. Then
$$
\underset{\deg(\fr)\to \infty}{\lim}\left(\frac{\#
X^R_o(\F_{q^2})}{g(X^R_o)}\right)=q-1.
$$
\end{thm}
\begin{proof}
By comparing the expression for the genus in Theorem \ref{thmGen}
with the expressions for the number of supersingular points in
Corollary \ref{thmSS}, one concludes that $\lim\inf$ of the sequence
in question is greater or equal to $(q-1)$. (We know that all
supersingular points are rational over $\F_{(T)}^{(2)}\cong
\F_{q^2}$.) On the other hand, the $\lim\sup$ of the same sequence
is bounded from above by $(q-1)$ according to (\ref{VDbound}). The
claim follows.
\end{proof}

Even though the series $\{X^R_o\}_R$ is asymptotically optimal over
$\F_o^{(2)}$, the individual curves fail to have a particularly
large number of rational points when $\deg(\fr)$ is small. The
reason for this is that the number of supersingular points on
$X^R_o$ is roughly $(q_o-1)g(X^R_o)$ for any $R$, whereas the Weil
bound for a curve over $\F_{o}^{(2)}$ is roughly $2g(X)q_o$. The
Weil bound is known to be quite close to the best possible when the
genus is relatively small, hence it is not surprising to find that
the known maximal number of points on curves of genus $g(X^R_o)$ is
approximately twice as large as the provable number of points on
$X^R_o$. We give a table for comparison when $q=2,3$ and genus $\leq
50$.

In the tables below we consider the curves $X^R_o$ when $\# R=2$ and
$o=(T)$. The first column indicates the degrees of the places in
$R$, the second and the third columns give the genus and the number
of supersingular points on $X^R_o$, respectively. The forth and the
fifth columns indicate the known maximum number of
$\F_{q^2}$-rational points on a curve of genus $g(X^R)$ and the best
known theoretic upper bound, respectively; these numbers are taken
from \cite{vdGvdV}. \vfill

\vspace{0.1in}
\begin{center}
\begin{tabular}{cccccc}
\hline\hline & &  & \# supersingular &
max \# of $\F_{q^2}$-rational & upper\\
$\mathbf{q=2}$ & $R$ & $g(X^R)$ & points &  points known & bound\\
\hline\\
& $(1,2)$ & 2 & 1& 10& 10\\
& $(1,3)$ & 2 & 5& 10& 10\\
& $(1,4)$ & 6 & 5& 20& 20\\
& $(1,5)$ & 10 & 13& 27& 27\\
& $(1,6)$ & 22 & 21& 42& 48\\
& $(1,7)$ & 42 & 45& 75& 80\\
\hline\\
& $(2,2)$ & 4 & 3& 15& 15\\
& $(2,3)$ & 8 & 7& 21& 24\\
&$(2,4)$ & 16 & 15& 36& 38\\
& $(2,5)$ & 32 & 31& 57& 65\\
\hline\\
& $(3,3)$ & 16 & 19& 36& 38\\
& $(3,4)$ & 36 & 35& 64& 71\\
\hline\hline
\end{tabular}
\end{center}

\vspace{0.2in}
\begin{center}
\begin{tabular}{cccccc}
\hline\hline  & &  & \# supersingular &
max \# of $\F_{q^2}$-rational & upper\\
$\mathbf{q=3}$ & $R$ & $g(X^R)$ & points &  points known & bound\\
\hline\\
& $(1,1)$ & 0 & 4& 4& 4\\
& $(1,2)$ & 3 & 4& 28& 28\\
& $(1,3)$ & 6 & 16& 35& 40\\
& $(1,4)$ & 21 & 40& 88& 95\\
\hline\\
& $(2,2)$ & 9 & 16& 48& 50\\
& $(2,3)$ & 27 & 52& 104& 114\\
\hline\hline
\end{tabular}
\end{center}



\end{document}